\documentclass[11pt,leqno]{article}     
\usepackage[utf8]{inputenc}
\usepackage{amssymb,amsmath,latexsym,mathtools}
\usepackage[english]{babel}
\usepackage{dsfont}
\usepackage{hyperref}
\usepackage{parskip}
\usepackage{enumitem}
\numberwithin{equation}{section}  
\usepackage{amsthm}     
\theoremstyle{plain}
\newtheorem{theorem}{Theorem}
\newtheorem*{theorem-non}{Theorem}
\newtheorem{corollary}[theorem]{Corollary}
\newtheorem{lemma}[theorem]{Lemma}

\newtheorem{proposition}[theorem]{Proposition}
\numberwithin{theorem}{section}
\theoremstyle{definition}

\newtheorem{conjecture}[theorem]{Conjecture}
\newtheorem*{conjecture-non}{Conjecture}
\theoremstyle{remark}

\usepackage{tikz-cd}

\title{Supersingular primes and Bogomolov property}          
\author{Soumyadip Sahu\\ School of Mathematics, Tata Institute of Fundamental Research\\ 1 Homi Bhabha Road, Mumbai, 400005, Maharashtra, India.\\ Email: soumyadip.sahu00@gmail.com}       
\date{}

\begin{document}
\maketitle
\begin{abstract}
Let $E$ be an elliptic curve over a number field $K$ with at least one real embedding and $L$ be a finite extension of $K$. We generalize a result of Habegger to show that $L(E_{\text{tor}})$, the field generated by the torsion points of $E$ over $L$, has the Bogomolov property. Moreover, the N\'eron-Tate height on $E\big(L(E_{\text{tor}})\big)$ also satisfies a similar discreteness property. Our main tool is a general criterion of Plessis that reduces the problem to the existence of a supersingular prime for $E$ satisfying certain conditions. 
\end{abstract}



\section{Statement of the result}
\label{section1}
The absolute logarithmic Weil height (abbreviated as \textit{height}) of an algebraic number is a nonnegative real number \cite{heights} that measures its arithmetic complexity. By Northcott's theorem, there are only finitely many algebraic numbers with bounded degrees and bounded heights. Kronecker's theorem tells us that the height of an algebraic number is zero if and only if it is zero or a root of unity. As a consequence, given a number field, there exists $\epsilon > 0$ so that the height of each element of the field, except zero and roots of unity, is larger than $\epsilon$. This phenomenon inspires the following definition. We say that an algebraic extension $L/\mathbb{Q}$ has the \textit{Bogomolov property} if there exists $\epsilon > 0$ such that each element of $L^{\times}$ having height $< \epsilon$ is a root of unity. The property's name is motivated by a similar conjecture of Bogomolov regarding points of small height on algebraic curves of genus $\geq 2$. In recent years, many number theorists have investigated the Bogomolov property for a diverse class of infinite algebraic extensions arising from arithmetic and geometric sources \cite{property-B}. For example, Amoroso and Zannier \cite{abelianfield1}, building upon earlier work of Amoroso and Dvornicich, showed that $K^{\text{ab}}$, the maximal abelian extension of a number field $K$, has the Bogomolov property. One can also generalize the notion of Bogomolov property to the elliptic curves by replacing Weil height with the N\'eron-Tate height. In more detail, let $E$ be an elliptic curve defined over a number field $K$. Given an algebraic extension $L/K$, we say that the group $E(L)$ has the \textit{Bogomolov property} if there exists $\epsilon > 0$ so that each point on $E(L)$ having N\'eron-Tate height $<  \epsilon$ is a torsion point. Several authors have studied the elliptic Bogomolov property to illustrate the similarities between the arithmetic of the torus and the elliptic curve. For instance, Silverman \cite{silverman} showed that if $E$ is an elliptic curve over a number field $K$, then $E(K^{\text{ab}})$ has the Bogomolov property. Our point of interest is a striking result of Habegger \cite{habegger} which proves both the toric and elliptic Bogomolov property for the field generated by the torsion points of an elliptic curve. 

We fix an algebraic closure $\overline{\mathbb{Q}}$ of $\mathbb{Q}$ and work with the subfields of $\overline{\mathbb{Q}}$. Set $\mu_{\infty}$ to be the collection of roots of unity in $\overline{\mathbb{Q}}$. For an elliptic curve $E$ over a number field $K$, let $E_{\text{tor}}$ denote the torsion subgroup of $E(\overline{\mathbb{Q}})$. Now suppose $L/K$ is an algebraic extension. Then $L(E_{\text{tor}})$ is the field generated over $L$ by the coordinates of the torsion points of $E$ for a $K$-rational Weierstrass model. This field does not depend on the choice of the Weierstrass model.  

\begin{theorem}
\label{theorem1.1}
\emph{(Habegger)} Let $E$ be an elliptic curve defined over $\mathbb{Q}$.  Then $\mathbb{Q}(E_{\emph{tor}})$ has the Bogomolov property. Moreover, the group $E\big(\mathbb{Q}(E_{\emph{tor}})\big)$ possesses the elliptic Bogomolov property.  
\end{theorem} 

Amoroso and Terracini \cite{galrep} explored the possibility of formulating Theorem~\ref{theorem1.1} for the infinite extensions attached to a compatible family of $\ell$-adic Galois representations. In particular, one can ask if the following generalization of Habbeger's result holds: 

\begin{conjecture-non}
Let $E$ be an elliptic curve over a number field $K$.  Then $K(E_{\text{tor}})$ and $E(K(E_{\text{tor}}))$ have the Bogomolov property.
\end{conjecture-non}

Note that if $L/K$ is a finite extension, then one can view $E$ as an elliptic curve over $L$ to obtain a statement for $L(E_{\text{tor}})$ from this conjecture. Therefore, the conjecture above is equivalent to the following stronger version:

\begin{conjecture}
\label{conjecture1.2}
Let $E$ be an elliptic curve over a number field $K$ and $L/K$ be a finite extension.  Then $L(E_{\text{tor}})$ and $E(L(E_{\text{tor}}))$ have the Bogomolov property.
\end{conjecture}  

If $L$ is an algebraic extension with the Bogomolov property and $F/L$ is a finite extension, then $F$ need not have the Bogomolov property; see \cite{property-B}. For an elliptic curve with complex multiplication, Conjecture~\ref{conjecture1.2} is an easy consequence of the aforementioned results of Amoroso-Zannier and Silverman regarding the maximal abelian extension (Proposition~\ref{proposition2.2}). The main point of Habegger's work is to deal with the non-CM case by careful use of Galois theory and equidistribution theorems. In a subsequent development, Frey \cite{frey22} generalized Habegger's arguments to establish the field extension part of Conjecture~\ref{conjecture1.2} for the elliptic curves over $\mathbb{Q}$. Her work also provides an explicit lower bound for the Weil height in terms of a suitable supersingular prime of the elliptic curve. More recently, Plessis \cite{plessis} improved Frey's work to derive a very general criterion for the Bogomolov property of the field generated by torsion points of an elliptic curve over a number field which plays a key role in our work.  

Let $L/K$ be a finite extension of number fields. Given a finite place $v$ of $K$, we write $e_w(L|K)$, resp. $f_w(L|K)$, for the ramification index, resp. residue degree, of a place $w$ of $L$ with $w \mid v$. Moreover, set \[d_v(L) = \max_{w \mid v} [L_w:K_v].\]  
We use $h(\cdot)$ to denote the Weil height of an algebraic number. 

\begin{theorem}\emph{(Plessis)}
\label{theorem1.3}
Let $E$ be an elliptic curve defined over a number field $K$. Suppose that $L/K$ is a finite Galois extension. Now suppose $v$ is a place of $K$ lying over a rational prime $p$ so that 
\begin{enumerate}[label=(\roman*), align=left, leftmargin=0pt]
\item $E$ has a supersingular reduction at $v$ and $j_E \not\equiv 0, 1728 (\emph{mod }v)$,
\item $p > \max\{3, 2d_v(L)\}$ and the image of the canonical Galois representation 
\begin{equation}
\label{1.1}
\emph{Gal}\big(L(E[p])/L\big) \to \emph{GL}_2(\mathbb{Z}/p\mathbb{Z})	
\end{equation}
contains $\emph{SL}_2(\mathbb{Z}/p\mathbb{Z})$,
\item $e_v(K \mid \mathbb{Q}) = 1$ and $f_v(K \mid \mathbb{Q}) \leq 2$. 
\end{enumerate} 
Then, for all $\alpha \in L(E_{\emph{tor}})^{\times}  \backslash \mu_{\infty}$, we have 
\[h(\alpha) \geq \frac{1}{4^{p^2d_v(L)}+1}\Big(\frac{\log p}{d_v(L)(40\sqrt{2}+2)[K:\mathbb{Q}]p^{2p^2d_v(L)+2}}\Big)^{2 + \frac{4}{p^{p^2d_v(L)/4}-2}}.\]
Moreover, if 
\begin{enumerate}[label=(\roman*), align=left, leftmargin=0pt]
\setcounter{enumi}{3}
\item $v$ is unramified in $L$,
\item the representation in (ii) is surjective,
\end{enumerate} 
then $E\big(L(E_{\emph{tor}})\big)$ has the Bogomolov property. 
\end{theorem} 
  
An obvious obstruction to directly applying this theorem is to find a supersingular prime satisfying the desiderata. Nevertheless, Elkies \cite{elkies} proved that if the base field $K$ has at least one real embedding, then $E$ is supersingular at infinitely many primes of $K$. Plessis uses this result to generalize Theorem~\ref{theorem1.1} to the elliptic curves over a real quadratic field. The goal of this short note is to demonstrate that the third condition in Theorem~\ref{theorem1.3} poses no difficulty at all if we study the problem for an elliptic curve over its minimal field of definition, namely $\mathbb{Q}(j_E)$, where $j_E$ is the $j$-invariant of $E$. Thus, one can combine the results of Serre and Plessis to obtain the following special case of Conjecture~\ref{conjecture1.2}. A version of this result was previously announced and proved under an additional assumption in the author's master's thesis \cite{thesis}.

\begin{theorem}
\label{theorem1.4}
Let $E$ be an elliptic curve over a number field $K$ so that $E$ has infinitely many supersingular primes. Then Conjecture~\ref{conjecture1.2} holds for $E$. 
\end{theorem}

We also provide an explicit lower bound for the height on $L(E_{\text{tor}})$ in terms of a supersingular prime whenever the theorem holds (Proposition~\ref{proposition2.4}).
One can use Elikes's infinitude result to directly deduce the conjecture for a large class of elliptic curves from the theorem above.  

\begin{corollary}
\label{corollary1.5}
Let $E$ be an elliptic curve over a number field $K$. Assume that $\mathbb{Q}(j_E)$ has at least one real embedding. Then the elliptic curve $E$ satisfies Conjecture~\ref{conjecture1.2}. In particular, if $K$ itself has at least one real embedding, then Conjecture~\ref{conjecture1.2} holds for $E$.  
\end{corollary} 

There is no general result regarding the existence of supersingular primes for the elliptic curves over totally imaginary fields. For instance, Elkies \cite{elkies87} considered the elliptic curve
\[\text{$E: iy^2 = x^3+ (i-2)x^2 + x$ over $K = \mathbb{Q}(i)$ with $j_E = \frac{2^{14}}{i - 4}$,}\]
and verified using a computer that $E$ has no supersingular prime lying over an odd prime less than $74000$. However, he provides a heuristic argument to conjecture that $E$ should have $C\log \log x$ many supersingular primes up to $x$ for some $C > 0$ as $x \to \infty$. If this conjecture is true, then Theorem~\ref{theorem1.4} implies Conjecture~\ref{conjecture1.2} for the curve $E$ described above.

\section{Details of the argument}
\label{section2}
Let $E$ be an elliptic curve over a number field $K$ with $j$-invariant $j_E$. We know that $E$ admits a model over $\mathbb{Q}(j_E)$ and any other field having this property contains $\mathbb{Q}(j_E)$ \cite[p.45]{silvermanbook}. For convenience, write $K_{E} := \mathbb{Q}(j_E)$ and let $\mathcal{O}_{K_E}$ denote the ring of integers of $K_E$. Our main point is as follows:

\begin{lemma}
\label{lemma2.1}
Let $E$ be an elliptic curve defined over $K_E$. Then, there exists a finite subset $S$ of prime ideals of $\mathcal{O}_{K_E}$ containing the primes ramified in the extension $K_E/\mathbb{Q}$ so that each supersingular prime outside $S$ has residue degree $1$ or $2$ over $\mathbb{Q}$.  
\end{lemma}  

\begin{proof}
Write $j_E\mathcal{O}_{K_E} = \mathfrak{a}\mathfrak{b}^{-1}$ where $\mathfrak{a}$ and $\mathfrak{b}$ are nonzero coprime integral ideals. Put $N = \lvert \mathcal{O}_{K_E} /\mathfrak{b}\rvert$. Then $\mathfrak{b} \mid N\mathcal{O}_{K_E}$ and $Nj_E \in \mathcal{O}_{K_E}$. Observe that $\mathbb{Z}[Nj_E]$, the subring generated by $Nj_E$, is an order inside $K_E$, i.e., $\mathbb{Z}[Nj_E]$ is a finite index $\mathbb{Z}$-submodule of $\mathcal{O}_{K_E}$. Define $S$ to be the collection of places of $\mathcal{O}_{K_E}$ consisting of the prime ideals dividing the rational integer \[ND_{K_E}[\mathcal{O}_{K_E}: \mathbb{Z}[Nj_E]]\] where $D_{K_E}$ is the discriminant of $K_E/\mathbb{Q}$ and $[\mathcal{O}_{K_E}: \mathbb{Z}[Nj_E]]$ is the index of the sublattice $\mathbb{Z}[Nj_E]$ in $\mathcal{O}_{K_E}$. Now suppose $\mathfrak{p} \notin S$ and $p$ is the rational prime lying below $\mathfrak{p}$. By construction $p$ is coprime to $[\mathcal{O}_{K_E}: \mathbb{Z}[Nj_E]]$. Thus one can apply $- \otimes_{\mathbb{Z}} \mathbb{Z}/p\mathbb{Z}$ to the inclusion $\mathbb{Z}[Nj_E] \rightarrow \mathcal{O}_{K_E}$ to discover that
\[\mathcal{O}_{K_E}/p\mathcal{O}_{K_E} = \mathbb{Z}/p\mathbb{Z}[\overline{Nj_E}].\]
Here $\overline{Nj_E}$ is the image of $Nj_E$ under $\mathcal{O}_{K_E} \twoheadrightarrow \mathcal{O}_{K_E}/p\mathcal{O}_{K_E}$ and $\mathbb{Z}/p\mathbb{Z}[\overline{Nj_E}]$ is the $\mathbb{Z}/p\mathbb{Z}$-subalgebra generated by $\overline{Nj_E}$. Since $p$ does not ramify in $K_E$, we have $p\mathcal{O}_{K_E} = \prod_{\mathfrak{P} \mid p} \mathfrak{P}$ where the product ranges over primes in $\mathcal{O}_{K_E}$. As a consequence, $\mathcal{O}_{K_E}/p\mathcal{O}_{K_E}\cong \prod_{\mathfrak{P} \mid p} \mathcal{O}_{K_E}/\mathfrak{P}$ as $\mathbb{Z}/p\mathbb{Z}$-algebra.
Projecting onto the factor corresponding to $\mathfrak{p}$, we see that \[\mathcal{O}_{K_E}/\mathfrak{p} = \mathbb{Z}/p\mathbb{Z}[\widetilde{Nj_E}]\]
where $\widetilde{Nj_E}$ refers to the image under the reduction $\mathcal{O}_{K_E} \twoheadrightarrow \mathcal{O}_{K_E}/\mathfrak{p}$ and $\mathbb{Z}/p\mathbb{Z}[\widetilde{Nj_E}]$ is the $\mathbb{Z}/p\mathbb{Z}$-subalgebra generated by $\widetilde{Nj_E}$.  Again, the choice of $\mathfrak{p}$ ensures that $j_E$ is integral at $\mathfrak{p}$. Let $\mathcal{O}_{K_E, \mathfrak{p}}$ be the localization of $\mathcal{O}_{K_E}$ at $\mathfrak{p}$ and $\widetilde{j_E}$ be the image of $j_E$ under $\mathcal{O}_{K_E,\mathfrak{p}} \twoheadrightarrow \mathcal{O}_{K_E}/\mathfrak{p}$. Since $p \nmid N$, it is clear that 
$\mathbb{Z}/p\mathbb{Z}[\widetilde{Nj_E}] = \mathbb{Z}/p\mathbb{Z}[\widetilde{j_E}]$. Now suppose $\mathfrak{p}$ is a prime of supersingular reduction for the elliptic curve $E$. Then $E$ has a good reduction at $\mathfrak{p}$ and $\widetilde{j_E}$ equals the $j$-invariant of the reduced elliptic curve. Therefore $\widetilde{j_E}$ lies in a quadratic extension of $\mathbb{Z}/p\mathbb{Z}$ (Theorem~3.1 in \cite[V.3]{silvermanbook}). But $\mathcal{O}_{K_E}/\mathfrak{p} = \mathbb{Z}/p\mathbb{Z}[\widetilde{j_E}]$. Hence the degree of the extension $\mathcal{O}_{K_E}/\mathfrak{p}$ over $\mathbb{Z}/p\mathbb{Z}$ must be $\leq 2$. It follows that any supersingular prime $\mathfrak{p}$ of $E$ that lies outside our set $S$ satisfies $f_{\mathfrak{p}}(K_E|\mathbb{Q}) \leq 2$ as desired.    
\end{proof}

Lemma~\ref{lemma2.1} allows us to apply Plessis's criterion to an elliptic curve with infinitely many supersingular primes provided the Galois representation  \eqref{1.1} has a big image at a large supersingular prime. Thus, it is convenient to treat the CM case separately where the latter condition fails due to abelian behavior.

\begin{proposition}
\label{proposition2.2}
Let $E$ be an elliptic curve with complex multiplication defined over a number field $K$, and $L/K$ be a finite extension. Then Conjecture~\ref{conjecture1.2} holds $E$. Moreover, for each $\alpha \in L(E_{\emph{tor}})^{\times} \backslash \mu_{\infty}$, we have
\[h(\alpha) > 3^{-4d^2-4d-6}\]
where $d = [L: \mathbb{Q}]$.  
\end{proposition}

\begin{proof}
Assume that $E$ has complex multiplication by an order in $\mathbb{Q}(\sqrt{-D})$ for some positive integer $D$. Set $K_D = K(\sqrt{-D})$ and $L_{D} = L(\sqrt{-D})$. Then, the theory of complex multiplication asserts that $K_D(E_{\text{tor}}) \subseteq K_D^{\text{ab}}$. In particular, $L(E_{\text{tor}}) \subseteq L_D(E_{\text{tor}}) \subseteq L_D^{\text{ab}}$. Therefore, an application of the results of Amoroso-Zannier \cite{abelianfield1} and Silverman \cite{silverman} stated in Section~\ref{section1} shows that $L(E_{\text{tor}})$ and $E\big(L(E_{\text{tor}})\big)$ have the Bogomolov property. For the explicit lower bound for Weil height, we need a quantitative refinement of the result used above. Let $F$ be a number field of degree $n$ and let $\alpha \in \overline{\mathbb{Q}}^{\times} \backslash \mu_{\infty}$ be an algebraic number so that $F(\alpha)/F$ is abelian. Amoroso and Zannier, in a subsequent paper \cite{abelianfield2}, have obtained the following explicit bound: 
\begin{equation}
\label{2.1}
h(\alpha) > 3^{-n^2-2n-6}.
\end{equation}
We take $F = L_D$. Thus $n = [L_D:\mathbb{Q}] \leq 2d$ and the inequality in the statement is a consequence of \eqref{2.1}. 
\end{proof}

Before proceeding towards the proof of Theorem~\ref{theorem1.4}, we present a simple observation to clarify the hypothesis of the theorem. 

\begin{lemma}
\label{lemma2.3}
Let $E$ be an elliptic curve over a number field $K$ and $L/K$ be a finite extension. Then, $E/K$ has infinitely many supersingular primes if and only if $E/L$ also has infinitely many supersingular primes.  
\end{lemma}

\begin{proof}
For an elliptic curve over a finite field, the definition of supersingularity \cite[p.145]{silvermanbook} concerns only its properties over algebraic closure. Therefore, if $E/K$ has supersingular reduction at a prime $\mathfrak{p}$ of $K$, then $E/L$ has supersingular reduction at each prime $\mathfrak{P}$ of $L$ satisfying $\mathfrak{P} \mid \mathfrak{p}$. Note that in this situation, $E/L$ automatically has a good reduction at each $\mathfrak{P}$ \cite[p.197]{silvermanbook}. Conversely, if $E/L$ is supersingular at a prime $\mathfrak{P}$ of $L$, then $E/K$ is supersingular at $\mathfrak{p} = \mathfrak{P} \cap K$ provided $E$ has a good reduction at $\mathfrak{p}$. Since any elliptic curve has only finitely many primes of bad reduction, the lemma follows from these two observations. 
\end{proof}

\paragraph{Proof of Theorem~\ref{theorem1.4}.} Let the notation be as in the statement of the theorem. We have already dealt with the CM case in Proposition~\ref{proposition2.2}. Assume that $E$ does not have complex multiplication. As before, write $K_E = \mathbb{Q}(j_E)$, and let $E_0$ be a model of $E$ over the field $K_E$. For concreteness, fix an isomorphism $\varphi: E_0 \xrightarrow{\cong} E$ defined over a finite extension $K'/K$. In the light of Lemma~\ref{lemma2.3}, the hypothesis of this theorem ensures that $E/K'$ has infinitely many supersingular primes. Thus, one can use the isomorphism $\varphi$ along with the descent statement in Lemma~\ref{lemma2.3} to conclude that $E_0/K_E$ also has infinitely many primes of supersingular reduction. Moreover, by Lemma~\ref{lemma2.1}, all but finitely many supersingular primes of $E_0/K_E$ have residue degree $1$ or $2$. Define $L'$ as the Galois closure of $K'L$ over $K_E$. We want to apply Plessis's criterion for the pair $(E_0/K_E, L'/K_E)$. Since $E_0$ does not have complex multiplication, a celebrated theorem of Serre \cite{serre} shows that the Galois representation \eqref{1.1} is surjective for all but finitely many rational primes $p$. Also, the parameter $d_v(L')$ is uniformly upper bounded by the degree of the extension $L'/K_E$. Thus, there exists a large supersingular prime $v$ so that (i)-(v) in the statement of Theorem~\ref{theorem1.3} holds for the pair $(E_0/K_E, L'/K_E)$. As a consequence $L'(E_{0, \text{tor}})$ and $E_0\big(L'(E_{0, \text{tor}})\big)$ have the Bogomolov property. To finish off the proof of the theorem, it suffices to check that $L'(E_{ \text{tor}})$ and $E\big(L'(E_{\text{tor}})\big)$ have the Bogomolov property. The first part is obvious since the isomorphism $\varphi$ induces an equality $L'(E_{0, \text{tor}}) = L'(E_{\text{tor}})$. Finally, the uniqueness of N\'eron-Tate height (Theorem~9.3 in \cite[VIII.9]{silvermanbook}) shows that the pullback of the N\'eron-Tate height on $E(\overline{\mathbb{Q}})$ to $E_0(\overline{\mathbb{Q}})$ via $\varphi$ equals the N\'eron-Tate height on $E_0(\overline{\mathbb{Q}})$. The preimage of $E\big(L'(E_{\text{tor}})\big)$ under $\varphi$ is precisely $E_0\big(L'(E_{\text{tor}})\big)$. Therefore $E\big(L'(E_{\text{tor}})\big)$ also has the Bogomolov property. \hfill $\square$   

\paragraph{Proof of Corollary~\ref{corollary1.5}.}
Let $E$ be an elliptic curve over a number field $K$ so that $K_E = \mathbb{Q}(j_E)$ has at least one real embedding. As before, it suffices to deal with the case when $E$ does not have complex multiplication. Recall the notation $E_0$, $\varphi$, and $K'$ from the proof of Theorem~\ref{theorem1.4}. Since $K_E$ has one real embedding, by the theorem of Elkies \cite{elkies} we know that $E_0/K_E$ has infinitely many supersingular primes. Now, one can use the isomorphism $\varphi$ and Lemma~\ref{lemma2.3} to deduce that $E/K'$, and hence $E/K$, has infinitely many supersingular primes. Thus, the assertion is a consequence of Theorem~\ref{theorem1.4}.

One can also write down an explicit lower bound for the Weil height in terms of a few relevant parameters. 

\begin{proposition}
\label{proposition2.4}
Let the notation be as in Theorem~\ref{theorem1.4} and assume that $E$ does not have complex multiplication. Let $v$ be a place of $K_E$ so that the pair $(E_0/K_E, L'/K_E)$ described in the proof of Theorem~\ref{theorem1.4} satisfies (i)-(iii) in Theorem~\ref{theorem1.3}. Then, for each $\alpha \in L(E_{\emph{tor}})^{\times} \backslash \mu_{\infty}$, we have 
\[h(\alpha) \geq \frac{1}{4^{p^2d_v(L')}+1}\Big(\frac{\log p}{d_v(L')(40\sqrt{2}+2)[K_E:\mathbb{Q}]p^{2p^2d_v(L')+2}}\Big)^{2 + \frac{4}{p^{p^2d_v(L')/4}-2}}.\]
\end{proposition}

The argument given in the proof of Theorem~\ref{theorem1.4} shows that there exist infinitely many places $v$ satisfying the hypothesis of the proposition.  

\begin{proof}
Follows from Theorem~\ref{theorem1.3} applied to the pair $(E_0/K_E, L'/K_E)$.   
\end{proof}

An explicit lower bound for N\'eron-Tate height seems more difficult due to the ineffectivity of the elliptic equidistribution result; see Remark~1.7 in \cite{plessis}.

\section*{Acknowledgments}
I wish to thank an anonymous referee for carefully reading the manuscript and suggesting numerous improvements regarding its presentation.

\end{document}